\documentclass[10p]{amsart}

\usepackage{graphicx}

\usepackage[autostyle]{csquotes}

\usepackage{proof}
\usepackage{savesym}
\savesymbol{cir}
\savesymbol{bibsection}
\usepackage{color}
\usepackage{mathtools}
\usepackage{bbm}
\usepackage{tikz-cd}
\usepackage{amssymb}
\usepackage{amsmath}
\usepackage{amsfonts}
\usepackage{amsthm}
\usepackage{float}
\usepackage{times}

\newcommand{\iv}[1]{{#1}^{-1}}

\newcommand{\sgr}[1]{\langle #1 \rangle} 
\newcommand{\csgr}[1]{{\rm C}[#1]}

\DeclareSymbolFont{forpolishl}{T1}{cmr}{m}{n}
\DeclareMathSymbol{\mathrmL}{0}{forpolishl}{'212}

\newcommand{\el}{$\ell$-}

\newcommand{\si}{\ensuremath{\sigma}}

\newcommand{\m}[1]{{\bf {#1}}}												
\newcommand{\vty}[1]{\ensuremath{\mathcal{#1}}}						

\newcommand{\cng}{{\rm Con}}	

\newcommand{\jn}{\vee}
\newcommand{\mt}{\wedge}

\newcommand{\rd}{\slash}
\newcommand{\ld}{\backslash}

\newcommand{\ut}{\textrm{\textup{e}}}

\newcommand{\eq}{\approx}

\newcommand{\abs}[1]{{ |{#1}|}}

\newcommand{\LG}{\mathcal{LG}}

\newcommand{\Rl}{\mathcal{RL}}

\newcommand{\Icansemrl}{\mathcal{S}em\,\mathcal{C}an\,\mathcal{I}\mathcal{RL}}

\newcommand{\Ni}{\mathcal{N}^{c\,}\mathcal{S}em\,\mathcal{C}an\,\mathcal{I}\mathcal{RL}}
\newcommand{\Nrl}{\mathcal{N}^{c\,}\mathcal{C}an\mathcal{RL}}
\newcommand{\Ntworl}{\mathcal{N}^{\,2}\mathcal{C}an\mathcal{RL}}

\newcommand{\LGcn}{\mathcal{LG}_{cn}}
\newcommand{\NLGcn}{\mathcal{N}^c\mathcal{LG}_{cn}}

\newcommand{\N}{ {\mathbb N} }
\newcommand{\Z}{ {\mathbb Z} }

\newtheorem{theorem}{Theorem}[section]
\newtheorem{lemma}[theorem]{Lemma}

\newtheorem{corollary}[theorem]{Corollary}

\newtheorem{proposition}[theorem]{Proposition}
\theoremstyle{definition}
\newtheorem{remark}[theorem]{Remark}

\newtheorem{example}[theorem]{Example}

\numberwithin{equation}{section}

\newcommand\nbd[1]{\protect\nobreakdash#1\hspace{0pt}}

\begin{document}

\title{Nilpotency and the Hamiltonian property for cancellative residuated lattices}

\author{Almudena Colacito}
\address{Laboratoire J. A. Dieudonn\'e, Universit\'e C\^ote d'Azur, Nice, France}
\email{almudena.colacito@unice.fr}
\thanks{The research reported here was funded by the Swiss National Science Foundation (SNF) grants No 200021\_165850,  No 200021\_184693, and No P2BEP2\_195473, and by the EU Horizon 2020 research and innovation programme under the Marie Sk{\l}odowska-Curie grant agreement No 689176.} 

\author{Constantine Tsinakis}
\address{Department of Mathematics, Vanderbilt University, Nashville, Tennessee, USA}
\email{constantine.tsinakis@vanderbilt.edu}





\begin{abstract}
The present article studies nilpotent and Hamiltonian cancellative residuated lattices and their relationship with nilpotent and Hamiltonian lattice-ordered groups.  In particular, results about lattice-ordered groups are extended to the domain of residuated lattices. The two key ingredients that underlie the considerations of this paper are the categorical equivalence between Ore residuated lattices and lattice-ordered groups endowed with a suitable modal operator; and Malcev's description of nilpotent groups of a given nilpotency class $c$ in terms of a semigroup equation.
\end{abstract}

\maketitle


\section{Introduction}\label{s:intro}

The present article studies nilpotent and Hamiltonian cancellative residuated lattices and their relationship with nilpotent and Hamiltonian lattice-ordered groups.  In particular, results about lattice-ordered groups (\el groups) are extended to the domain of residuated lattices. The two key ingredients that underlie the considerations of this paper are the categorical equivalence of~\cite{MT10}, which provides a new framework for the study of various classes of cancellative residuated lattices by viewing these structures as \el groups with a suitable modal operator; and Malcev's description~\cite{M53} (see also~\cite{NT63}) of nilpotent groups of a given nilpotency class $c$ in terms of a semigroup equation $\mathsf{L}_c$ (to be defined in Section \ref{s:Nilpotent}).

A plethora of evidence has been accumulated during the past two decades  demonstrating the fundamental importance of \el groups in the study of  algebras of logic\footnote{We use the term \emph{algebras of logic} to refer to residuated lattices---algebraic counterparts of propositional substructural logics---and their reducts. Substructural logics are non-classical logics that are weaker than classical logic, in the sense that they may lack one or more of the structural rules of contraction, weakening and exchange in their Gentzen-style axiomatization. They encompass a large number of non-classical logics related to computer science (linear logic), linguistics (Lambek Calculus), philosophy (relevant logics), and many-valued reasoning.}. For example, an essential result~\cite{Mun86} in the theory of MV-algebras is the categorical equivalence between the category of MV-algebras and the category of unital Abelian \el groups. Likewise, the non-commutative generalization of this result in~\cite{Dvu02} establishes a categorical equivalence between the category of pseudo MV-algebras and the category of unital \el groups. Further, the generalization of these two results in~\cite{MT10} shows that one can view GMV-algebras as \el groups with a suitable modal operator. The categorical equivalence in~\cite{MT10} mentioned above is another example in point.

In a complementary direction, the articles~\cite{LPT14, GFLT15, BKLT16, GFLT17, LPT18, BKT20, GFSTZ20} have  shown that large parts of the Conrad Program can be profitably extended to the much wider class of \emph{$\ut$\nbd{-}cyclic} residuated lattices, that is, those residuated lattices satisfying the equation $x\ld \ut\eq \ut\rd x$. The term \emph{Conrad Program} traditionally refers to P. Conrad's approach to the study of \el groups, which analyzes the structure of individual \el groups, or classes of \el groups, by means of an overriding inquiry into the lattice-theoretic properties of their lattices of convex \el subgroups. In the 1960s, Conrad's articles~\cite{Con60, Con61, Con65, Con68, Con69, Con73} pioneered this approach and demonstrated its usefulness. 

The present work builds on the aforementioned research. Nilpotent \el groups are the $\ell$\nbd{-}groups whose group reducts are (necessarily torsion-free) nilpotent groups. They share many important properties with Abelian \el groups, including representability (semilinearity) and the Hamiltonian property. In particular, they satisfy the congruence extension property.  The notion of Hamiltonian algebra arises as a generalization of the concept of Hamiltonian group~\cite{MR133391}. Borrowing the terminology from group theory, an \el group is said to be Hamiltonian if every convex \el subgroup is normal. Hamiltonian \el groups were first introduced implicitly in~\cite{Martinez1972}, and later studied extensively (see, e.g.,~\cite{Conrad1980, Reilly1983, MR1801994, MR2238913}). While Hamiltonian \el groups do not form a variety (\cite[Proposition 1.4]{Conrad1980}), a largest variety of Hamiltonian \el groups does exist and was identified in~\cite{Reilly1983}. A significant property of Hamiltonian \el groups is representability---namely, each Hamiltonian \el group is a subdirect product of totally ordered groups. Representability and the Hamiltonian property were established for nilpotent \el groups in~\cite{Kop75} (see also~\cite{Holl78} and~\cite{Reilly1983}, respectively). 

We conclude the introduction by illustrating the article's discourse. 
In Section~\ref{s:basics}, we dispatch some preliminaries on residuated lattices and their convex subuniverses. In Section~\ref{s:Nilpotent}, we study the quasivariety of submonoids of nilpotent $\ell$\nbd{-}groups. In particular, Theorem~\ref{t:five} shows that submonoids of nilpotent \el groups are precisely those nilpotent cancellative monoids that have unique roots. The theorem also provides a characterization for the quasivariety of submonoids of nilpotent cancellative residuated lattices. Its proof makes use of Theorem~\ref{t:catequiv}, which provides a bridge that connects nilpotent cancellative residuated lattices and nilpotent \el groups. 
The focus of Section~\ref{s:prelinearity} is the prelinearity property, with particular interest for some of its implications and equivalent formulations. We show in Theorem~\ref{t:prelinearity} that residuals in a prelinear residuated lattice preserve finite joins in the numerator, and convert finite meets to joins in the denominator. While prelinearity implies semilinearity in the presence of commutativity~\cite{MR1919685}, this is no longer the case for non-commutative varieties of residuated lattices. However, Theorem~\ref{t:prelinearity} shows that any prelinear cancellative residuated lattice has a distributive lattice reduct.

Section~\ref{s:Hamiltonian} is devoted to Hamiltonian residuated lattices. Theorem~\ref{t:hamiltrepres} shows that any Hamiltonian prelinear $\ut$-cyclic residuated lattice is semilinear, which implies that any Hamiltonian prelinear cancellative residuated lattice is semilinear (Corollary~\ref{c:hamsemil}). In particular, the result that Hamiltonian \el groups are representable is extended in Corollary~\ref{c:hamsemil} to prelinear cancellative residuated lattices. With these results at hand, we prove that there exists a  largest variety of Hamiltonian prelinear cancellative residuated lattices  (Theorem~\ref{t:hamvty}), thereby extending the corresponding result for \el groups. The main focus of Section~\ref{s:NilpotentPrel} is the class of nilpotent residuated lattices. First, nilpotent cancellative residuated lattices are proved to be Hamiltonian. As a consequence, nilpotent prelinear cancellative residuated lattices are semilinear. The arguments make use of the corresponding results for $\ell$\nbd{-}groups, by means of the categorical equivalence between nilpotent cancellative residuated lattices and nilpotent $\ell$\nbd{-}groups with a conucleus (see\ Theorem~\ref{t:catequiv}).

Given the role that semilinearity plays in the study of Hamiltonian and nilpotent prelinear cancellative varieties, the final section of the paper discusses varieties of semilinear cancellative residuated lattices. We show, {\em inter alia}, that any variety $\vty{V}$ of semilinear cancellative integral residuated lattices defined by monoid equations is generated by residuated chains whose monoid reducts are finitely generated objects in the quasivariety of monoid subreducts corresponding to $\vty{V}$ (Theorem~\ref{t:fgmongen}). 

\section{Residuated Lattices: Basic Concepts}\label{s:basics}

In this section we briefly recall some basic facts about residuated lattices and their structure; we refer to~\cite{BT03},~\cite{jip:res},~\cite{GJKO07}, and~\cite{MR2731976} for further details.

The set of positive natural numbers is $\N \coloneqq \{1, 2, \ldots\}$, and $\Z^+$ is the set $\N \cup \{0\}$. Throughout, by `poset' we mean `partially ordered set'. If $\mathcal{L}$ is a signature and $\mathcal{L}'\subseteq\mathcal{L}$, an $\mathcal{L}'$-algebra $\m{A}$ is an $\mathcal{L}'$-subreduct of an $\mathcal{L}$-algebra $\m{B}$ if $\m{A}$ is a subalgebra of the $\mathcal{L}'$-reduct of $\m{B}$. For simplicity, when $\mathcal{L}'$ is the monoid (resp.\ group, lattice or semilattice) signature, we sometimes refer to $\m{A}$ as a submonoid (resp.\ subgroup, sublattice or subsemilattice) of the $\mathcal{L}$-algebra $\m{B}$.

A \emph{residuated lattice} is an algebra $\m{L} = \langle L, \mt, \jn, \cdot, \ld, \rd, \ut \rangle$, where $\langle L, \cdot, \ut \rangle$ is a monoid, $\langle L, \mt, \jn \rangle$ is a lattice, and $\ld$ and $\rd$ are binary operations such that, for all $a, b, c \in {L}$, 
\begin{equation}\label{eq:residual1}
ab \le c \,\,\,\Longleftrightarrow\,\,\, a \le c \rd b \,\,\,\Longleftrightarrow\,\,\, b \le a \ld c,
\end{equation}
where $ab$ stands for the product $a\cdot b$, and $\le$ is the lattice order. We write $\ut$ for the monoid identity. The operations $\ld$ and $\rd$ are referred to as {\em left residual} and {\em right residual} of $\cdot$, respectively. We refer to $a$ as the {\em denominator} of $a \ld b$ (resp.\ $b \rd a$), and to $b$ as the {\em numerator} of $a \ld b$ (resp.\ $b \rd a$). Condition (\ref{eq:residual1}) is equivalent to $\cdot$ being order-preserving in each argument and, for every $a, b \in {L}$, the sets 
\begin{equation}\label{eq:residual2}
\{c \in {L} \mid a \cdot c \le b\}\quad\text{ and }\quad\{c \in {L} \mid c \cdot a \le b\}
\end{equation} 
containing greatest elements $a \ld b$ and $b \rd a$, respectively. Residuated lattices form a variety denoted by $\Rl$~\cite{bahls2003cancellative,BT03}. Throughout, we often write $t \le s$ for the equation $t \mt s \eq t$.

We recall here some relevant standard facts.

\begin{proposition}\label{prop:facts1}
The monoid operation $\cdot$ of any residuated lattice preserves all existing joins in each argument. The residuals $\ld$ and $\rd$ preserve all existing meets in the numerator, and convert existing joins in the denominator into meets. Consequently, residuals preserve order in the numerator, and reverse order in the denominator.
\end{proposition}

\begin{proposition}\label{prop:facts}
Every residuated lattice satisfies the equations 
\[
x \ld (y \rd z) \eq (x \ld y) \rd z, \quad x \rd yz \eq (x \rd z) \rd y, \quad xy \ld z  \eq y \ld (x \ld z),
\]
\end{proposition}

For any residuated lattice $\m{L}$, the set ${L}^- = \{a \in L \mid a \le \ut\}$ of negative elements of ${L}$ (including the monoid identity) is its \emph{negative cone}. It is the universe of a submonoid, and a sublattice of $\m{L}$, and it can be made into a residuated lattice, by defining $\ld_{\m{L}^-}$ and $\rd_{\m{L}^-}$ as 
\begin{equation}\nonumber
a \ld_{\m{L}^-} b \coloneqq a \ld b \mt \ut
\end{equation}
\begin{equation}\nonumber
\,\,a \rd_{\m{L}^-} b \coloneqq a \rd b \mt \ut,
\end{equation} 
for $a, b \in {L}^-$. Residuated lattices satisfying $x \land \ut \eq x$ are called {\em integral}. The class of integral residuated lattices can be equivalently defined relative to $\Rl$ by the equations 
\begin{equation}\label{eq:integrality}
x\ld\ut \eq \ut \eq \ut\rd x.
\end{equation}  

We call a residuated lattice \emph{cancellative} if its monoid reduct is a cancellative monoid. The class of cancellative residuated lattices is a variety (see~\cite[Lemma 2.5]{bahls2003cancellative}) defined relative to $\Rl$ by the equations: 
\begin{equation}\label{eq:cancellative}
xy\rd y \eq x \eq y\ld yx.
\end{equation} 

Replacing $x$ by $\ut$ in the preceding equation, we get:

\begin{proposition}\label{prop:xbelowx}
The equations $x \ld x \eq \ut \eq x \rd x$ hold in any cancellative residuated lattice.
\end{proposition}

A residuated lattice is said to be \emph{$\ut$-cyclic} if it satisfies the equation $x\ld\ut \eq \ut\rd x$.

\begin{proposition}\label{prop:cancecycl}
Any cancellative residuated lattice is $\ut$-cyclic.
\end{proposition} 

\begin{proof}
For any residuated lattice $\m{L}$ and $a \in {L}$, we have $a \ld (a \rd a) = (a \ld a) \rd a$ by Proposition~\ref{prop:facts}. Thus, by Proposition~\ref{prop:xbelowx}, if $\m{L}$ is cancellative, $a\ld\ut = \ut\rd a$ for every $a \in {L}$.
\end{proof}

If $\m{L}$ is a residuated lattice, we write $\mathcal{C}(\m{L})$ for the set of all convex subuniverses of $\m{L}$, ordered by set inclusion. Here, a {\em convex subuniverse} is an order-convex subuniverse of $\m{L}$. If $\m{L}$ is $\ut$-cyclic, $\mathcal{C}(\m{L})$ is a distributive lattice (see, e.g.,~\cite[Theorem 3.8]{BKLT16}). 

For any $S \subseteq {L}$, we write $\csgr{S}$ for the smallest convex subuniverse of $\m{L}$ containing $S$. As usual, we call $\csgr{S}$ the convex subuniverse {\em generated by} $S$, and write $\csgr{a}$ for $\csgr{\{a\}}$. We refer to $\csgr{a}$ as the {\em principal} convex subuniverse of $\m{L}$ generated by the element $a \in {L}$. If $\m{L}$ is a residuated lattice, and $a \in {L}$, the absolute value $\abs{a} \in {L}^-$ is defined as 
\[
a \mt (\ut\rd a) \mt \ut.
\]
Note that when $a \le \ut$, $\abs{a} = a$. 

If $S \subseteq {L}$, we write $\sgr{S}$ for the submonoid generated by $S$ in $\m{L}$. The following results are established in~\cite{BKLT16} (see\ Lemma 3.2, Corollary 3.3, and Lemma 3.6 in~\cite{BKLT16}).

\begin{lemma}\label{l:princconvex}
In any $\ut$-cyclic residuated lattice $\m{L}$, the followings hold:
\begin{itemize}
\item [{\rm (a)}] For any $S \subseteq {L}$, the convex subuniverse generated by $S$ is
\begin{align*}\smallskip
\csgr{S} = \csgr{\abs{S}} & = \{c \in {L} \mid t \le c \le t\ld \ut\text{, for some }t \in \sgr{\abs{S}}\} \\ 
	     & = \{c \in {L} \mid t \le \abs{c} \text{, for some }t \in \sgr{\abs{S}}\},
\end{align*}
where $\abs{S} \coloneqq \{\abs{s} \colon s \in S\}$. \smallskip

\item [{\rm (b)}] For any $a \in {L}$, the convex subuniverse generated by $a$ is
\begin{align*}\smallskip
\csgr{a} = \csgr{\abs{a}} & = \{c \in {L} \mid \abs{a}^n \le c \le \abs{a}^n\ld \ut \text{, for some }n \in \N\} \\ 
	     & =  \{c \in {L} \mid \abs{a}^n \le \abs{c} \text{, for some }n \in \N\}. 
\end{align*}
\item [{\rm (c)}] For any $a, b \in {L}$, $\csgr{\abs{a} \jn \abs{b}} = \csgr{a} \cap \csgr{b}$ and $\csgr{\abs{a} \mt \abs{b}} = \csgr{a} \jn \csgr{b}$.
\end{itemize}
\end{lemma}

If $\m{L}$ is a residuated lattice, and $a, b \in {L}$, we define 
\begin{equation}\label{eq:conjugates}
\lambda_{b}(a) \coloneqq (b \ld ab) \mt \ut \quad \text{ and } \quad \rho_{b}(a) \coloneqq (ba \rd b) \mt \ut,
\end{equation}
and refer to $\lambda_{b}(a)$ and $\rho_{b}(a)$ respectively as the {\em left} and {\em right conjugate} of $a$ by $b$. For any residuated lattice $\m{L}$, a convex subuniverse $H \in \mathcal{C}(\m{L})$ is said to be {\em normal} if for any $a \in {H}$ and any $b \in {L}$, $\lambda_b(a) \in {H}$ and $\rho_b(a) \in {H}$. It was proved in \cite[Theorem 4.12]{BT03} that the lattice $\mathcal{NC}(\m{L})$ of convex normal subuniverses of any residuated lattice $\m{L}$ is isomorphic to its congruence lattice $\cng{(\m{L})}$.

A {\em lattice-ordered group} (briefly, {\em \el group}) is an algebra $\m{G} = \langle G, \mt, \jn, \cdot, \iv{}, \ut \rangle$ such that $\langle G, \cdot, \iv{}, \ut \rangle$ is a group, $\langle G, \mt, \jn \rangle$ is a lattice, and the group operation distributes over the lattice operations. The class of \el groups is a variety. Here, it is identified with the term\nbd{-}equivalent subvariety $\LG$ of $\Rl$ defined by the equations
\[
x(x\ld \ut) \eq \ut \eq (\ut\rd x)x.
\]
The equivalence is given by $\iv{x} \coloneqq x \ld \ut = \ut \rd x$, $\iv{x}y \coloneqq x \ld y$, and $y\iv{x} \coloneqq y \rd x$. In any residuated lattice $\m{L}$, an element $a \in L$ is {\em invertible}, that is, it has a multiplicative inverse if 
\[
a(a \ld \ut) = \ut = (\ut \rd a)a.
\] 
Hence, the class of \el groups is identified with the class of those residuated lattices in which every element is invertible. 

Residuated lattices with a commutative monoid reduct are called {\em commutative} residuated lattices, and form a subvariety of $\Rl$. It is standard to call {\em Abelian} those $\ell$\nbd{-}groups whose underlying group is commutative. Here, a {\em monoid-subvariety of $\vty{V}$} is any variety defined relative to $\vty{V} \subseteq \Rl$ by monoid equations (e.g., commutative residuated lattices form a monoid-subvariety of $\Rl$). We also refer to $\vty{V}$ as a {\em monoid\nbd{-}variety}.

For any monoid $\m{M}$, we say that ${\le} \subseteq M \times M$ is a {\em partial order on $\m{M}$} if it is a partial order on $M$ and, for all $a, b, c, d \in M$, whenever $a \le b$, also $cad \le cbd$; if the order $\le$ is total, we call it a {\em total order on $\m{M}$}. If the total order is residuated, we say that $\m{M}$ {\em admits a residuated total order}, and we sometimes write $\langle \m{M}, \le \rangle$ for the resulting residuated lattice. It is immediate that any total order on (the monoid reduct of) a group is a residuated total order. Finally, a residuated lattice {\em admits a (residuated) total order} if its monoid reduct admits a residuated total order that extends its lattice order.


\section{Submonoids of Nilpotent Lattice-Ordered Groups}\label{s:Nilpotent}

The primary focus of this section is the quasivariety of submonoids of nilpotent $\ell$\nbd{-}groups. 
The main result of this section, Theorem~\ref{t:five}, provides a characterization of these monoids and, equivalently, of submonoids of nilpotent cancellative residuated lattices. In particular, a nilpotent monoid is a submonoid of a nilpotent \el group if and only if it is cancellative and has unique roots (in the sense to be defined below). 

A nilpotent group is one that has a finite central series. Given $c \in \N$, nilpotent groups of class $c$ (in short, $c$-nilpotent groups) are those possessing a central series of length at most $c$; they form a variety defined by the equation
\[
[[[x_1, x_2], \dots , x_c], x_{c+1}] \eq \ut.
\]
Thus, $1$-nilpotent groups coincide with Abelian groups, and every $c$-nilpotent group, $c \in \N$, is also $(c+1)$-nilpotent. 

Consider now the equation $\mathsf{L}_c \colon \, q_c(x,y,\bar{z}) \eq q_c(y,x,\bar{z})$, where $\bar{z}$ abbreviates a sequence of variables $z_1,z_2,\ldots\,$, and $q_c(x,y,\bar{z})$ is defined as follows, for $c \in \N$:
\begin{align*}
q_1(x,y,\bar{z}) & = xy \\ \smallskip
q_{c+1}(x,y,\bar{z}) & = q_{c}(x,y,\bar{z})z_{c}q_{c}(y,x,\bar{z}).
\end{align*}
The equation $\mathsf{L}_c$ characterizes $c$-nilpotent groups.

\begin{proposition}[\cite{NT63}, Corollary 1]\label{prop:nilplc}
A group is $c$-nilpotent if and only if it satisfies the equation $\mathsf{L}_c$.
\end{proposition}

\noindent We call a monoid {\em nilpotent of class $c$} (in short, $c$-nilpotent) if it satisfies $\mathsf{L}_c$, and call a residuated lattice {\em nilpotent of class $c$} (briefly, $c$-nilpotent) if its monoid reduct is $c$-nilpotent. The class of $c$-nilpotent residuated lattices is a monoid\nbd{-}variety of residuated lattices, and commutative residuated lattices coincide with $1$-nilpotent residuated lattices. We refer to a monoid (resp., residuated lattice) as {\em nilpotent} if it is $c$-nilpotent for some $c \in \N$.

\begin{example}
Clearly, every $c$-nilpotent $\ell$-group is a $c$-nilpotent residuated lattice. Similarly, the negative cone of a $c$-nilpotent $\ell$-group is a cancellative, integral $c$-nilpotent  residuated lattice. 
\end{example}

We present below an example of a commutative cancellative integral residuated chain, which is neither an $\ell$-group nor the negative cone of an $\ell$-group. 

\begin{example}
Let $\m{M}_1(x,y)$ be the free 2-generated commutative monoid over $\{x,y\}$. We consider the dual shortlex order, i.e., for words $t,s \in M_1(x,y)$ we have $t \le s$ iff $|t| > |s|$ or $|t|=|s|$ and $t <_{lex} s$ in the lexicographic order generated by $y < x$. For example,
\[
\ut > x > y > x^2 > xy > y^2 > x^3 > x^2y > xy^2 > y^3 > \dots
\]
Then $\m{M}_1(x,y)$ equipped with the considered order is an integral commutative cancellative residuated chain. However, it is not the negative cone of an $\ell$-group. Indeed, it is easy to see that the negative cone of an $\ell$-group satisfies the divisibility equation $(w\rd z)z \approx w \wedge z$ (see \cite[Corollary 6.3]{bahls2003cancellative}), whereas we have $(y \rd x)x=x^2 = y = x \wedge y$ in $\m{M}_1(x,y)$.
\end{example}

We also provide an example of a non-commutative nilpotent cancellative integral residuated chain, which is neither an $\ell$-group nor the negative cone of an $\ell$-group. 

\begin{example}
Let ${\m F}_2(x,y)$ be free 2\nbd{-}nilpotent group over $\{x,y\}$ and let ${\m S}_2(x,y)$ be the submonoid generated by $\{x,y\}$. The group ${\m F}_2(x,y)$ is isomorphic to the group ${\rm UT}_3(\Z)$ of unitriangular matrices (see \cite[Exercise 16.1.3]{KarM79}); the isomorphism, obtained by extending the variable assignment
\[
x\, \longmapsto \,
\begin{pmatrix}
1 \,\,\, & 0 \,\,\, & 0 \\
0 \,\,\, & 1 \,\,\, & 1 \\
0 \,\,\, & 0 \,\,\, & 1
\end{pmatrix}\,; \quad 
y\, \longmapsto \,
\begin{pmatrix}
1 \,\,\, & 1 \,\,\, & 0 \\
0 \,\,\, & 1 \,\,\, & 0 \\
0 \,\,\, & 0 \,\,\, & 1
\end{pmatrix}\,.
\]
can be also described on each element $x^{\alpha}y^{\beta}[x,y]^{\gamma}$ of $F_2(x,y)$, for $\alpha,\beta,\gamma \in \Z$, as follows:
\[
x^{\alpha}y^{\beta}[x,y]^{\gamma} \,\,\,\longmapsto\,\,\, 
\begin{pmatrix}
    1 \,\,\, & \beta \,\,\, & \gamma  \\
    0 \,\,\, & 1       \,\,\, & \alpha   \\
    0 \,\,\, & 0       \,\,\, & 1  
\end{pmatrix}\,\,
\]
\noindent 
Through this isomorphism, the monoid ${\m S}_2(x,y)$ is isomorphic to the submonoid of ${\rm UT}_3(\Z)$ whose underlying set is
\[
\{\,A \in {\rm UT}_3(\Z) \mid \alpha,\beta,\gamma \in \Z^+ \,\,\text{ and }\,\,\gamma \le \alpha\beta\,\}.
\]

\noindent We consider the following total order on ${\m S}_2(x,y)$ induced by the lexicographic order on the triples $( \alpha,\beta,\gamma )$: if we identify $a$ with $( \alpha_1,\beta_1,\gamma_1 )$ and $b$ with $( \alpha_2,\beta_2,\gamma_2 )$, define
\[
a \le^* b \quad \Longleftrightarrow \quad ( \alpha_1,\beta_1,\gamma_1 ) \ge_{lex} ( \alpha_2,\beta_2,\gamma_2 ).
\]
The monoid ${\m S}_2(x,y)$ equipped with this order is a nilpotent, cancellative, integral residuated chain that is neither an $\ell$-group nor the negative cone of an $\ell$-group. Indeed, as was noted above the negative cone of an $\ell$-group satisfies the law $(w \rd z)z \approx w \mt z$ (see \cite[Corollary 6.3]{bahls2003cancellative}), whereas here $(x \rd y)y = xy \neq x = x \mt y$.
\end{example}

A monoid $\m{M}$ is {\em right-reversible} if ${M}a \cap {M}b \neq \emptyset$, for all $a, b \in {M}$. A {\em group of (left) quotients} for a monoid $\m{M}$ is a group $\m{G}$ that has $\m{M}$ as a submonoid, and such that every $c \in {G}$ is of the form $c=\iv{a}b$ for some $a, b \in {M}$. By a classical result due to Ore (see, e.g.,~\cite[Section 1.10]{clifford1961algebraic},~\cite{dubreil1943}), a cancellative monoid $\m{M}$ has a group of quotients (unique up to isomorphism) if and only if $\m{M}$ is right-reversible.

We call a right-reversible cancellative monoid {\em Ore}, and write $\m{G}(\m{M})$ for its group of quotients. Further, we call a residuated lattice {\em Ore} if its monoid reduct is Ore. 

\begin{proposition}[\cite{NT63}, Theorem 1]\label{prop:nilp}
A cancellative monoid has a $c$-nilpotent group of quotients if and only if it satisfies the equation $\mathsf{L}_c$.
\end{proposition}

\noindent The preceding result implies in particular that all nilpotent cancellative residuated lattices are Ore. 

The categorical equivalence in~\cite{MT10} provides a bridge between nilpotent cancellative residuated lattices and nilpotent \el groups. Recall that a function $\sigma \colon \m{P} \to \m{P}$ on a poset $\m{P} = \langle P, \le \rangle$ is a {\em co-closure operator} if it is order-preserving ($x \le y$ entails $\sigma(x) \le \sigma(y)$), contracting ($\sigma(x) \le x$), and idempotent ($\sigma(\sigma(x)) =\sigma(x)$). The image of $\sigma$ will be denoted by $\m{P}_{\sigma}$. We say that a co-closure operator $\sigma$ on a poset $\m{P}$ is a {\em conucleus} if $\sigma(\ut) = \ut$ and $\sigma(x)\sigma(y) \le \sigma(xy)$. If $\m{L} = \langle L, \mt, \jn, \cdot, \ld, \rd, \ut \rangle$ is a residuated lattice and $\sigma$ a conucleus on $\m{L}$, then the image $\m{L}_{\sigma}$ is a join-subsemilattice and a submonoid of $\m{L}$. It can be made into a residuated lattice, with operations $\mt_{\sigma}$, $\ld_{\sigma}$, and $\rd_{\sigma}$, defined by
\[
a \mt_{\sigma} b \coloneqq \sigma(a \mt b)\,, \quad \quad
a \ld_{\sigma} b  \coloneqq  \sigma(a \ld b)\,, \quad \quad
a \rd_{\sigma} b  \coloneqq  \sigma(a \rd b),
\]
for any $a,b \in {L}_{\sigma}$ (see~\cite[Lemma 3.1]{MT10}).

Let $\LGcn$ be the category with objects $\langle \m{G}, \si \rangle$ consisting of an \el group $\m{G}$ augmented with a conucleus $\si$ such that the underlying group of the \el group $\m{G}$ is the group of quotients of the monoid reduct of $\si[\m{G}]$. The morphisms of $\LGcn$ are \el groups homomorphisms that commute with the conuclei. The category $\mathcal{ORL}$ of Ore residuated lattices and residuated lattice homomorphisms was shown to be equivalent to $\LGcn$~\cite[Theorem 4.9]{MT10}. The results collected here suffice to provide a restriction of this equivalence to the category $\Nrl$ of $c$-nilpotent cancellative residuated lattices and residuated lattice homomorphisms, and the full subcategory $\NLGcn$ of $\LGcn$ consisting of objects whose first component is a $c$-nilpotent \el group.

We will not make use of the full categorical equivalence, but keep in mind the following key idea: Every nilpotent cancellative residuated lattice $\m{L}$ (of class $c$) `sits' inside a uniquely determined nilpotent \el group $\m{G}(\m{L})$ (of class $c$) as a submonoid, and as a join-subsemilattice. Further, $\m{L}$ can be seen as the image of $\m{G}(\m{L})$ relative to a suitable conucleus.

\begin{theorem}[\cite{MT10}, Lemmas 4.2 - 4.3 - 4.4]\label{t:catequiv}
Let $\m{L}$ be a $c$-nilpotent cancellative residuated lattice. If $\le$ denotes the partial order of $\m{L}$, then the binary relation ${\preceq} \subseteq \m{G}(\m{L}) \times \m{G}(\m{L})$ defined, for $a,b,c,d \in {L}$, by
\[
\iv{a}b \preceq \iv{c}d\text{ iff there exist }m,n \in {L}\text{ such that }mb \le nd\text{ and }ma=nc,
\]
is the unique partial order on $\m{G}(\m{L})$ that extends $\le$ and makes $\m{G}(\m{L})$ into a $c$-nilpotent \el group. Further, the map 
\[
\sigma_{\m{L}} \colon \m{G}(\m{L}) \to \m{G}(\m{L})\,; \quad \quad \sigma_{\m{L}}(\iv{a}b) = a \ld b, \,\,\,\text{for all }a,b \in {L},
\] 
is a conucleus on $\m{G}(\m{L})$ and $\m{L} = \m{G}(\m{L})_{\si_{\m{L}}}$.
\end{theorem}

The main result of this section, Theorem~\ref{t:five} below, characterizes those cancellative monoids that embed into nilpotent \el groups, and into nilpotent cancellative residuated lattices. Before we proceed with its proof, we recall a few pertinent properties of nilpotent groups. In what follows, a monoid $\m{M}$ (or a group) is said to have {\em unique roots} if, whenever $a,b \in M$, and $a^n = b^n$ for some $n \in \N$, then $a=b$.

\begin{lemma}[\cite{KarM79}, Theorems 16.2.3 - 16.2.7 - 16.2.8]\label{l:karg}
The following properties hold in any nilpotent group $\m{G}$.
\begin{itemize}
\item [{\rm (a)}] Every nontrivial normal subgroup of $\m{G}$ intersects the center nontrivially.\smallskip

\item [{\rm (b)}] The set of torsion elements of $\m{G}$ is a normal subgroup of $\m{G}$.\smallskip

\item [{\rm (c)}] If $\m{G}$ is torsion-free, it has unique roots.
\end{itemize}
\end{lemma}

For any variety $\vty{V}$ of residuated lattices, we write $\vty{M}(\vty{V})$ for the class of monoid subreducts of $\vty{V}$, that is, those monoids that are submonoids of the monoid reduct of a residuated lattice from $\vty{V}$. That $\vty{M}(\vty{V})$ is always a quasivariety is readily seen by checking that it is closed under ultraproducts, submonoids and direct products. 

\begin{theorem}\label{t:five}
For any monoid $\m{M}$, the following are equivalent:
\begin{enumerate}
\item $\m{M}$ is a submonoid of a $c$-nilpotent \el group.\smallskip

\item $\m{M}$ is $c$-nilpotent, cancellative, and has unique roots.\smallskip

\item $\m{M}$ has a group of quotients $\m{G(M)}$, that is $c$-nilpotent and torsion-free.\smallskip

\item $\m{M}$ is a submonoid of a totally ordered $c$-nilpotent group.\smallskip

\item $\m{M}$ is a submonoid of a $c$-nilpotent cancellative residuated lattice.
\end{enumerate}
\end{theorem}

\begin{proof}
For (1) $\Rightarrow$ (2), assume that $\m{M}$ is a submonoid of a $c$-nilpotent $\ell$\nbd{-}group $\m{G}$. That $\m{M}$ is $c$-nilpotent is immediate by Proposition~\ref{prop:nilplc}. It remains to show that $\m{M}$ has unique roots. To this end, suppose that $a^n = b^n$ for some $n \in \N$, and $a, b \in {M}$. Then, $a^n = b^n$ in $\m{G}$. Now, since $\m{G}$ is an \el group, it is torsion-free, and by Lemma~\ref{l:karg}(c), $a = b$. 

For (2) $\Rightarrow$ (3), observe that $\m{G(M)}$ exists and is $c$-nilpotent by Proposition~\ref{prop:nilp}. Suppose now that $(\iv{a}b)^n=\ut$, for some $a \neq b \in {M}$, and $n \in \N$. Then, $\iv{a}b$ is in the torsion subgroup of $\m{G(M)}$, which is normal by Lemma~\ref{l:karg}(b). By Lemma~\ref{l:karg}(a), its intersection with the center of $\m{G}(\m{M})$ is non\nbd{-}trivial, and hence, there exists a central element $\iv{c}d \in {G(M)}$ such that $c \neq d \in {M}$, and $(\iv{c}d)^m = \ut$ for some $m \in \N$. As $\iv{c}d$ is a central element of $\m{G}(\m{M})$, $c(\iv{c}d) =(\iv{c}d)c$ or, equivalently, $d\iv{c} = \iv{c}d$. Therefore, an easy induction on $m \in \N$ shows that 
\[
(\iv{c}d)^m = (\iv{c})^md^m = \ut.
\]
This implies $c^m = d^m$, which contradicts the assumption that $\m{M}$ has unique roots, since $c$ and $d$ are assumed to be distinct.

For (3) $\Rightarrow$ (4), it suffices to observe that $\m{G(M)}$ admits a total order, as it is torsion\nbd{-}free and nilpotent (see \ \cite[Theorem 2.2.4]{BMR1977}). 

Now, (4) $\Rightarrow$ (5) is trivial, as any totally ordered $c$-nilpotent group is a $c$-nilpotent cancellative residuated lattice. 

Finally, we show (5) $\Rightarrow$ (1). By assumption $\m{M}$ is a submonoid of a $c$-nilpotent cancellative residuated lattice $\m{L}$. Let $\m{G(L)}$ be the $\ell$\nbd{-}group of quotients of $\m{L}$, as defined in Theorem~\ref{t:catequiv}. Since $\m{L}$ is a submonoid of $\m{G(L)}$, the result follows.
\end{proof}

We write $\m{M}_c(X)$ for the free object over $X$ in the quasivariety $\vty{M}(\Nrl)$, and ${\m S}_c(T)$ for the submonoid of the free $c$-nilpotent group $\m{F}_c(X)$ generated by any subset $T \subseteq F_c(X)$. 

\begin{proposition}\label{prop:fisleftquot}
For any $c \in \N$ and any set $X$, the monoid ${\m S}_c(X)$ is isomorphic to the free object $\m{M}_c(X)$ over $X$ in the quasivariety $\vty{M}(\Nrl)$. Further, the free $c$-nilpotent group $\m{F}_c(X)$ is isomorphic to the group of quotients of $\m{M}_c(X)$.
\end{proposition}

\begin{proof}
First, observe that ${\m S}_c(X)$ is a member of $\vty{M}(\Nrl)$, by Theorem~\ref{t:five}. Therefore, the unique monoid homomorphism 
\[
\gamma \colon \m{M}_c(X) \to {\m S}_c(X)
\] 
extending the identity map on $X$ exists by the universal property of $\m{M}_c(X)$. The map $\gamma$ is clearly onto, since ${\m S}_c(X)$ is generated by $X$ as a monoid. Further, observe that $\m{M}_c(X)$ sits as a submonoid inside a $c$-nilpotent group $\m{H}$, and that there exists a unique group homomorphism 
\[
\delta \colon \m{F}_c(X) \to \m{H}
\] 
extending the identity map on $X$. This map restricts to a surjective monoid homomorphism 
\[
\hat{\delta} \colon {\m S}_c(X) \to \m{M}_c(X),
\] 
since $\m{M}_c(X)$ is generated by $X$. Thus, $\gamma$ and $\delta$ are inverses to each other. Finally, the second part of the statement follows from the fact that the group of quotients of the monoid ${\m S}_c(X)$ is the free $c$-nilpotent group $\m{F}_c(X)$. This is because any group generated by an Ore monoid $\m{M}$ is a group of quotients of $\m{M}$ (see, e.g.,~\cite[Section 1.10]{clifford1961algebraic}). 
\end{proof}

\noindent Let us remark that Proposition~\ref{prop:fisleftquot} could also be obtained as a consequence of a more general result that can be found in~\cite[\S 5]{MR3183390}.


\section{Prelinearity and its Implications}\label{s:prelinearity}

The remainder of the paper will be concerned with classes of prelinear residuated lattices. A residuated lattice is said to be \emph{prelinear} if it satisfies the following equations:
\[
\begin{array}{rllcrll}
{\sf (LPL)} & (x \ld y \mt \ut) \jn (y \ld x \mt \ut) \eq \ut & \text{ and } & {\sf (RPL)} & (x \rd y \mt \ut) \jn (y \rd x \mt \ut) \eq \ut.
\end{array}
\]
This section is devoted to exploring prelinearity, with particular interest for some of its implications and equivalent formulations. More precisely, Theorem~\ref{t:prelinearity} below shows that residuals in a prelinear residuated lattice preserve finite joins in the numerator, and convert finite meets to joins in the denominator. While prelinearity implies semilinearity in the presence of commutativity~\cite{MR1919685}, this is no longer the case in non-commutative settings. However, Theorem~\ref{t:prelinearity} shows that any prelinear cancellative residuated lattice has a distributive lattice reduct. 

We start with a preliminary lemma.

\begin{lemma}[\cite{bahls2003cancellative}, Proposition 4.1]\label{l:distributive}
The followings are equivalent for any lattice $\m{L}$.

\begin{itemize}
\item [\rm (1)] $\m{L}$ is distributive.\smallskip

\item [\rm (2)] For all $a, b \in L$ with $a \le b$, there exists a join-endomorphism $f \colon L \to L$ such that $f(b) = a$ and $f(x) \le x$, for all $x \in L$.
\end{itemize}
\end{lemma}

The distributivity of the lattice reduct of any $\ell$-group is an immediate consequence of Lemma~\ref{l:distributive}---it suffices to take $f(x) = x\iv{b}a$.

Consider the following pairs of equations:
\[
\begin{array}{rllcrll}
{\sf (LPL2)} & (y \mt z) \ld x \eq (y \ld x) \jn (z \ld x) & \text{ and } & {\sf (LPL3)} & x \ld (y \jn z) \eq (x \ld y)  \jn (x \ld z); \\[.075in]
{\sf (RPL2)} & x \rd (y \mt z) \eq (x \rd y) \jn (x \rd z) & \text{ and } & {\sf (RPL3)} & (y \jn z) \rd x \eq (y \rd x) \jn (z \rd x).
\end{array}
\]
Parts of the next result can be found in~\cite[Proposition 6.10]{BT03}, and~\cite[Corollary 4.2]{bahls2003cancellative}.

\begin{theorem}\label{t:prelinearity}
The followings hold for any residuated lattice $\m L$.
\begin{itemize}
\item [{\rm (a)}] ${\sf (LPL)}$ implies ${\sf (LPL2)}$ and ${\sf (LPL3)}$. \smallskip

\item [{\rm (b)}] If $\m L$ satisfies $\ut \mt (y \jn z) \eq (\ut \mt y) \jn (\ut \mt z)$, the equations ${\sf (LPL)}$, ${\sf (LPL2)}$ and ${\sf (LPL3)}$ are equivalent. \smallskip

\item [{\rm (c)}] ${\sf (LPL3)}$ and $x \ld x \eq \ut$ imply distributivity of the lattice reduct; in particular, any prelinear cancellative residuated lattice has a distributive lattice reduct.

\end{itemize}
\end{theorem}

\begin{proof}
For (a), consider a residuated lattice $\m L$ satisfying ${\sf (LPL)}$. For any $a, b, c\in {L}$, 
\[
(b \mt c) \ld a \ge (b \ld a) \jn (c \ld a).
\]
To obtain the reverse inequality, and hence conclude ${\sf (LPL2)}$, it suffices to show that
\[
\ut \le [(b \ld a) \jn (c \ld a)] \rd  [(b \mt c) \ld a].
\]
Let $u = (b \ld a) \jn (c \ld a).$  Then, we have 
\begin{align*}
u \rd [(b \mt c)\ld a] &\ge^{\rm (1)} (b\ld a) \rd [(b \mt c) \ld a] \\                                             
  			      &=^{\rm (2)}  b\ld [a\rd  [ (b \mt c) \ld a]] \\
 			      &\ge^{\rm (3)} b \ld(b \mt c) \\
  			      &=^{\rm (4)} (b\ld c)\mt(c\ld c) \\
  			      &\ge^{\rm (5)} (b\ld c)\mt\ut,
\end{align*}
\noindent where (1), (3), (4), and (5) follow by (\ref{eq:residual1}) -- (\ref{eq:residual2}), and by Proposition~\ref{prop:facts1}, while (2) follows by Proposition~\ref{prop:facts}.  Likewise, 
\[
u \rd [(b \mt c)\ld a] \ge (c\ld b)\mt\ut.
\]
\noindent Hence, 
\[
u \rd [(b \mt c)\ld a] \ge [(b\ld c)\mt\ut]\jn [ (c\ld b)\mt\ut] = \ut,
\]
as was to be shown.

For ${\sf (LPL3)}$, observe that it is always the case that
\[
(a \ld b) \jn (a \ld c) \le a \ld(b \jn c).
\]
To establish the reverse inequality, we show that
\[
[a \ld (b \jn c)] \ld [(a \ld b) \jn (a \ld c)] \ge\ut.
\]
Let $u = (a \ld b) \jn (a \ld c)$. We have
\begin{align*}
[a \ld (b \jn c)] \ld u &\ge^{\rm (1)} [a \ld (b \jn c)]\ld (a \ld b) \\
  			     &=^{\rm (2)} [a(a \ld (b \jn c))] \ld b \\
  			     &\ge^{\rm (3)} (b \jn c) \ld b \\
  			     &=^{\rm (4)} (b\ld b)\mt(c\ld b) \\
  			     &\ge^{\rm (5)} (c\ld b)\mt\ut
\end{align*}
where (1), (3), (4), and (5) follow by (\ref{eq:residual1}) -- (\ref{eq:residual2}), and by Proposition~\ref{prop:facts1}, while (2) follows by Proposition~\ref{prop:facts}.
Likewise, 
\[
[a \ld (b \jn c)] \ld u \ge (b \ld c) \mt \ut.
\]
Consequently,
\[ 
[a \ld (b \jn c)] \ld u \ge [(c \ld b)\mt\ut] \jn [(b \ld c)\mt\ut] = \ut,
\]
and thence, the conclusion. 

For (b), assume $\m L$ satisfies ${\sf (LPL2)}$, and let $a, b, c\in L$. Then,
\begin{align*}
 [(a\ld b)\mt\ut]\jn [(b\ld a)\mt\ut] &=^{\rm (1)} [a\ld (a\mt b)\mt\ut]\jn [b\ld (a\mt b)\mt\ut] \\
                           &=^{\rm (2)} [(a\ld (a\mt b))\jn (b\ld (a\mt b))]\mt\ut \\
                           &=^{\rm (3)} [(a\mt b)\ld (a\mt b)]\mt\ut \\
                           &\ge^{\rm (4)}  \ut\mt\ut=\ut,
 \end{align*}
where (1) and (4) follow by (\ref{eq:residual1}) -- (\ref{eq:residual2}), and by Proposition~\ref{prop:facts1}, the equality (2) follows by the assumption, and (3) is a consequence of ${\sf (LPL2)}$.

Finally, assume $\m L$ satisfies ${\sf (LPL3)}$, and let $a, b, c\in {L}$. Then
\begin{align*}
 [(a\ld b)\mt\ut]\jn [(b\ld a)\mt\ut] &=^{\rm (1)}     [(a\jn b)\ld b)\mt\ut]\jn [(a\jn b)\ld a)\mt\ut]\\   
                            &=^{\rm (2)}     [((a\jn b)\ld b)\jn ((a\jn b)\ld a)]\mt\ut\\  
                            &=^{\rm (3)}     [(a\jn b)\ld (a\jn b)]\mt\ut\\ 
                            &\ge^{\rm (4)}  \ut\mt\ut=\ut,
\end{align*}
where (1) and (4) follow by (\ref{eq:residual1}) -- (\ref{eq:residual2}), and by Proposition~\ref{prop:facts1}, the equality (2) follows by the assumption, and (3) is a consequence of ${\sf (LPL3)}$.

For (c), assume that ${\sf (LPL3)}$ and $x \ld x \eq \ut$ hold in $\m{L}$. For any $a \le b \in {L}$, define 
\[
f \colon {L} \to {L}, \quad\quad f(x) = a(b \ld x).
\] 
Then, that $f$ is a join-endomorphism follows from
\begin{align*}
a(b \ld (x \jn y)) & =^{\rm (1)}  a((b \ld x) \jn (b\ld y)) \\
& =^{\rm (2)}  a(b \ld x) \jn a(b \ld y),
\end{align*}
where (1) follows by ${\sf (LPL3)}$, and (2) by Proposition~\ref{prop:facts1}. Further, we have 
\[
f(b) = a(b \ld b) = a
\]
by assumption, and $f(x) \le x$ since 
\[
a \le b \,\,\, \Longrightarrow^{\rm (3)} \, b \ld x \le a \ld x \,\,\, \Longrightarrow^{\rm (4)} \, a(b \ld x) \le x,
\]
where we get (3) by Proposition~\ref{prop:facts1}, and (4) by (\ref{eq:residual1}). We conclude by Lemma~\ref{l:distributive}.
\end{proof}

\noindent Even though Theorem~\ref{t:prelinearity} is presented here only for ${\sf (LPL)}$, ${\sf (LPL2)}$, ${\sf (LPL3)}$, the dual arguments show the analogous results for the equations ${\sf (RPL)}$, ${\sf (RPL2)}$, ${\sf (RPL3)}$. More precisely, the equations ${\sf (RPL)}$, ${\sf (RPL2)}$, ${\sf (RPL3)}$ are equivalent under the hypothesis of Theorem~\ref{t:prelinearity}(b). Further, ${\sf (RPL3)}$ and $x \rd x \eq \ut$ entail distributivity of the lattice reduct.

Following the proof of Theorem~\ref{t:prelinearity}(c), it is easy to see that every prelinear integral residuated lattice has a distributive lattice reduct, as it satisfies ${\sf (LPL3)}$ and $x \ld x \eq \ut$. Finally, in the case of cancellative (resp.\ integral) residuated lattices, prelinearity is equivalent to ${\sf (LPL3)}$ and ${\sf (RPL3)}$. The left-to-right direction is immediate from Theorem~\ref{t:prelinearity}(a). For the converse, observe that ${\sf (LPL3)}$ and cancellativity (resp.\ integrality) together entail distributivity of the lattice reduct. Therefore, by Theorem~\ref{t:prelinearity}(b), ${\sf (LPL)}$ must hold.


\section{Prelinearity and Cancellativity: the Hamiltonian Case}\label{s:Hamiltonian}

This section is devoted to residuated lattices whose convex subuniverses are normal. A residuated lattice $\m{L}$ is said to be {\em Hamiltonian} if every convex subuniverse $H$ of $\m{L}$ is normal, and {\em semilinear} if $\m{L}$ is the subdirect product of totally ordered residuated lattices (we sometimes use the term `(residuated) chain' to denote a totally ordered residuated lattice). A variety $\vty{V}$ of residuated lattices is Hamiltonian if every member of $\vty{V}$ is Hamiltonian and semilinear if each subdirectly irreducible member of $\vty{V}$ is totally ordered\footnote{It is standard to call {\em representable} those \el groups that are semilinear.}.

The result that Hamiltonian \el groups are representable is extended here to prelinear $\ut$-cyclic residuated lattices. More precisely, Theorem~\ref{t:hamiltrepres} shows that ${\sf (LPL)}$ and ${\sf (RPL)}$ provide an axiomatization for semilinearity relative to any variety of Hamiltonian $\ut$\nbd{-}cyclic residuated lattices. Later, this is used to show that a largest variety of Hamiltonian prelinear cancellative residuated lattices exists, thereby extending the analogous result for \el groups. 

\begin{proposition}[\cite{BKLT16}, Theorem 5.6]\label{prop:quasisemil}
For any residuated lattice $\m{L}$, the following are equivalent:
\begin{itemize}
\item [\rm (1)] $\m{L}$ is semilinear.\smallskip

\item [\rm (2)] $\m{L}$ is prelinear, and it satisfies the quasiequation: 
\begin{equation}\label{eq:semilinear}
x \jn y \eq \ut \Longrightarrow \lambda_u(x) \jn \rho_v(y) \eq \ut.
\end{equation}
\end{itemize}
\end{proposition}


\noindent The laws ${\sf (LPL)}$ and ${\sf (RPL)}$ hold in all totally ordered residuated lattices and hence in all semilinear residuated lattices. For Hamiltonian $\ut$-cyclic residuated lattices, the converse also holds.

\begin{theorem}\label{t:hamiltrepres}
Any Hamiltonian prelinear $\ut$-cyclic residuated lattice is semilinear.
\end{theorem}

\begin{proof}
Let $\m{L}$ be a Hamiltonian $\ut$-cyclic residuated lattice satisfying the prelinearity laws, and suppose $a \lor b = \ut$, for $a, b \in L$. Then, 
\[
\ut = \csgr{a \lor b} = \csgr{a} \cap \csgr{b}
\]
by Lemma~\ref{l:princconvex}. Since $\m{L}$ is Hamiltonian, for any $c,d \in L$, we have $\lambda_c(a) \in \csgr{a}$, and $\rho_d(b) \in \csgr{b}$. Therefore, again by Lemma~\ref{l:princconvex},
\begin{align*} \smallskip
\csgr{\lambda_c(a) \lor \rho_d(b)} & =  \csgr{\lambda_c(a)} \cap \csgr{\rho_d(b)} \\ \smallskip
& \subseteq  \csgr{a} \cap \csgr{b} \\ 
& =  \ut,
\end{align*}
and hence, $\lambda_c(a) \lor \rho_d(b) = \ut$.
\end{proof}

\noindent Theorem~\ref{t:hamiltrepres} implies the result in~\cite{MR1919685} that a prelinear commutative residuated lattice is semilinear.

\begin{corollary}\label{c:hamsemil}
Any commutative prelinear residuated lattice is semilinear.
\end{corollary}

Moreover, by Theorem~\ref{t:hamiltrepres} we also conclude the following.

\begin{corollary}\label{c:hamsemil}
Any Hamiltonian prelinear cancellative residuated lattice is semilinear.
\end{corollary}

\begin{proof}
The conclusion follows by Proposition~\ref{prop:cancecycl} and Theorem~\ref{t:hamiltrepres}.
\end{proof}

The class of Hamiltonian \el groups is not itself an equational class. The variety of {\em weakly Abelian} \el groups, introduced in~\cite{Martinez1972}, is the largest variety of Hamiltonian $\ell$\nbd{-}groups~\cite[Corollary 2.3]{Reilly1983}. It is defined relative to $\LG$ by the equation
\begin{equation}\label{eq:hamiltlg}
(x \mt \ut)^2 \le \iv{y}(x \mt \ut)y.
\end{equation}
We extend this result to the context of prelinear cancellative residuated lattices. Note that the analogous result fails for $\ut$-cyclic residuated lattices (see\ \cite[Theorem 6.3]{BKLT16}), as there is no largest variety of Hamiltonian $\ut$-cyclic residuated lattices. 

\begin{theorem}\label{t:hamvty}
There exists a largest variety of Hamiltonian prelinear cancellative residuated lattices. More precisely, a variety $\vty{V}$ of prelinear cancellative residuated lattices is Hamiltonian if and only if $\vty{V}$ satisfies the equation
\begin{equation}\label{eq:hamilt}
(x \mt \ut)^2  \le \lambda_y(x) \wedge \rho_z(x),
\end{equation}
where $\lambda_y$ and $\rho_z$ are defined as in (\ref{eq:conjugates}).
\end{theorem}

\begin{proof}
Suppose that $\vty{V}$ is a variety of prelinear cancellative residuated lattices that satisfies equation~\eqref{eq:hamilt}. 
Let $\m{L} \in \vty{V}$, $H \in \mathcal{C}(\m{L})$, $a \in {H}$, and $b \in {L}$. We have $(a \mt \ut)^2 \in {H}$ and
\[
(a \mt \ut)^2 \le \lambda_b(a) \wedge \rho_b(a) \le (b \ld ab) \mt \ut \le \ut \,\text{, } \quad (a \mt \ut)^2 \le \lambda_b(a) \wedge \rho_b(a) \le (ba \rd b) \mt \ut \le \ut.
\]
Hence the convexity of $H$ implies that $\lambda_b(a), \rho_b(a) \in {H}$. We have shown that $\m{H}$ is normal and hence $\vty{V}$ is a  Hamiltonian variety.

To prove the converse direction, we use logical contrapositive. Suppose that $\vty{V}$ is a variety of prelinear cancellative residuated lattices that fails the equation~(\ref{eq:hamilt}), that is, in $\vty{V}$ either $(x \mt \ut)^2  \not\le \lambda_y(x)$ or $(x \mt \ut)^2  \not\le \rho_z(x)$. We may assume without loss of generality that $(x \mt \ut)^2  \not\le \lambda_y(x)$. Then, by Corollary~\ref{c:hamsemil}, there exists a residuated chain $\m{T} \in \vty{V}$ and an element $a \in {T}^-$ such that $a^2 \not \le b \ld ab \mt \ut$ for some $b \in T$ or, by cancellativity,
\begin{equation}\label{eq:assumption}
ab < ba^2\text{ for some }b \in {T}.
\end{equation} 
Condition~\eqref{eq:assumption} can be used to construct a non-Hamiltonian member $\m{L} \in \vty{V}$. Note first the following:

\paragraph{Claim.} For any $n \in \N$, $a^n b < ba^{2n}$.
\\ \\
\noindent Indeed, we proceed by induction on $n \in \N$. The base case follows from (\ref{eq:assumption}). For the induction step, observe that
\begin{align*}
 a^{n+1}b & =\,\,\,\,\,\,  aa^nb \\                                                
                & <^{\rm (1)}  aba^{2n} \\                                                          
                & <^{\rm (2)}  ba^2a^{2n} \\
                & =\,\,\,\,\,\,  ba^{2(n+1)},
\end{align*}
where (1) follows by the induction hypothesis, and (2) from (\ref{eq:assumption}).

\paragraph{Claim.} For any $n \in \N$, $ab^n < b^na^{2^n}$.
\\ \\
\noindent Again, we proceed by induction on $n \in \N$. The base case follows from (\ref{eq:assumption}). For the induction step, observe that
\begin{align*}
 ab^{n+1} & =\,\,\,\,\,\,  ab^nb \\                                                
                & <^{\rm (1)}  b^na^{2^n}b \\                                                          
                & <^{\rm (2)}  b^nba^{2\cdot2^n} \\
                & =\,\,\,\,\,\,  b^{n+1}a^{2^{(n+1)}},
\end{align*}
where (1) follows by the induction hypothesis, and (2) from the previous claim.
\\ \\
\noindent In view of the cancellativity of $\m{T}$, the inequality $ab^n < b^n a^{2n}$ may be written as $b^n \ld ab^n < a^{2n}$, which implies that
\begin{equation}\label{eq:counterex}
b^n \ld ab^n < a^{n},\,\,\, \text{ for all }\,n \in \N.
\end{equation}
To conclude the argument, consider now
\[
\m{L} = \prod_{i \,\in \,\Z^+} \m{T}_i,
\] 
where $\m{T}_i$ is a copy of $\m{T}$ for every $i \in \Z^+$. Let $\bar{a}, \bar{b} \in {L}$ be the elements $\bar{a}(i) = a$ and $\bar{b}(i) = b^i$, for all $i \in \Z^+$. It is now clear, in view of~\eqref{eq:counterex}, that $(\bar{a})^n \not\le \lambda_{\bar{b}}(\bar{a})$ for all $n \in \N$, as $\lambda_{\bar{b}}(\bar{a})(i) = b^i \ld ab^i$ for all $i \in \Z^+$. Then, $\lambda_{\bar{b}}(\bar{a}) \not \in \csgr{\bar{a}}$, and $\lambda_{\bar{b}}(\bar{a})$ witnesses the failure of the Hamiltonian property for $\m{L}$. 
\end{proof}


\section{Prelinearity and Cancellativity: the Nilpotent Case}\label{s:NilpotentPrel}

The preceding section demonstrates that Hamiltonian prelinear cancellative residuated lattices bear striking similarities with Hamiltonian \el groups. We now move into the study of nilpotent prelinear cancellative residuated lattices. It is known that nilpotent $\ell$\nbd{-}groups are representable (\cite{Kop75}; see\ \cite[Theorem 4]{Holl78}), and even Hamiltonian (\cite[Theorem 2.4]{Reilly1983}; see\ \cite[Corollary 2]{Kop75}). The main result of this section is Theorem~\ref{t:nilprlham}, where nilpotent cancellative residuated lattices are proved to be Hamiltonian. As a consequence, we obtain that nilpotent prelinear cancellative residuated lattices are semilinear.

\begin{theorem}\label{t:nilprlham}
Every nilpotent cancellative residuated lattice is Hamiltonian.
\end{theorem}

\begin{proof}
Let $\m{L}$ be a nilpotent cancellative residuated lattice. By Theorem~\ref{t:hamvty}, it suffices to show that $\m{L}$ satisfies the equation (\ref{eq:hamilt}). For this, pick $c, d \in {L}$, with $c \le \ut$. Then, both $dc^2 \preceq cd$ and $c^2d \preceq dc$ hold in $\m{G}(\m{L})$, since the latter is a nilpotent, and hence Hamiltonian, \el group. Since $\m{L}$ is a submonoid of $\m{G}(\m{L})$, and the restriction of the order $\preceq$ to $\m{L}$ is the order $\le$ of $\m{L}$, then $dc^2 \le cd$ and $c^2d \le dc$ hold in $\m{L}$. Therefore, using the equations (\ref{eq:cancellative}) we can conclude that $\m{L}$ satisfies $c^2 \le d \ld cd$ and $c^2 \le dc \rd d$, for $c,d \in {L}$ with $c \le \ut$. Thus, for all $a, b_1, b_2 \in {L}$, 
\[
(a \mt \ut)^2 \le b_1 \ld (a \mt \ut)b_1 \le (b_1 \ld ab_1) \mt \ut\,\,\,\text{ and }\,\,\,(a \mt \ut)^2 \le b_2(a \mt \ut) \rd b_2 \le (b_2a \rd b_2) \mt \ut,
\]
that is, $(a \mt \ut)^2  \le \lambda_{b_1}(a) \wedge \rho_{b_2}(a)$, as was to be shown.
\end{proof}

\begin{remark}
For the variety $\Ntworl$, we can also provide a direct argument, without going through Theorem~\ref{t:catequiv}. Pick any $\m{L} \in \Ntworl$. Then, for $a, b \in {L}$, 
\begin{align*}\smallskip
b(a \mt \ut) \ut (a \mt \ut)b & =^{\rm (1)} (a \mt \ut)b \ut b(a \mt \ut) \\ \smallskip
& \le^{\rm (2)}  b^2 (a \mt \ut) \\ \smallskip
& \le^{\rm (3)}  b^2a,
\end{align*}
where (1) follows by the equation ${\sf L}_2$, and (2) and (3) follow from $(a \mt \ut) \le a, \ut$.
Therefore, 
\begin{align*}\smallskip
(a \mt \ut)^2b & =^{\rm (4)} b \ld b(a \mt \ut)^2b \\
& \le^{\rm (5)} b \ld b^2 a \\
& =^{\rm (6)} ba,
\end{align*}
where (4) and (6) follow by (\ref{eq:cancellative}), and (5) follows by what we showed above, together with Proposition~\ref{prop:facts1}. That $b(a \mt \ut)^2 \le ab$ can be proved symmetrically.
\end{remark}

We can conclude from Theorem~\ref{t:hamiltrepres} and Theorem~\ref{t:nilprlham} that nilpotent prelinear cancellative residuated lattices are semilinear. For the convenience of the reader, we also present an alternative argument, making use of Theorem~\ref{t:catequiv}.

\begin{theorem}\label{t:nilsemil}
Every nilpotent prelinear cancellative residuated lattice is semilinear.
\end{theorem}

\begin{proof}
Let $\m{L}$ be a nilpotent prelinear cancellative residuated lattice and let $\m{G}(\m{L})$ be its $\ell$\nbd{-}group of quotients. We show that $\m{L}$ satisfies (\ref{eq:semilinear}) of Proposition~\ref{prop:quasisemil}. Let $a, b, c \in {L}$, and assume $a \jn b = \ut$. By Theorem~\ref{t:catequiv}, $a \jn b = \ut$ holds in the nilpotent \el group $\m{G}(\m{L})$. This implies that $\iv{c}ac \jn b = \ut$ by Proposition~\ref{prop:quasisemil}, as nilpotent \el groups are representable. Hence, also $ac \jn cb = c$ in $\m{L}$, and therefore, 
\[
c \ld (ac \jn cb) = c\ld c = \ut.
\] 
Now, by Theorem~\ref{t:prelinearity}(a), we get $c \ld ac \jn c \ld cb = \ut$, that is, $\lambda_c(a) \jn b= \ut$. Similarly, and by the comment after Theorem~\ref{t:prelinearity}, we can conclude $\lambda_c(a) \jn  \rho_d(b)= \ut$.
\end{proof}

This article focuses on cancellative varieties of residuated lattices, as one of the main tools used here is the categorical equivalence between cancellative (Ore) residuated lattices and $\ell$-groups with a conucleus. However, it is reasonable to ask whether some of the results presented here remain true in the absence of cancellativity. E.g., can we extend Corollary~\ref{c:hamsemil} to the nilpotent case, thereby concluding that prelinearity axiomatizes semilinearity in all nilpotent varieties? A positive answer to this question would require proof-techniques beyond the  ones used in this paper.


\section{Ordering Integral Residuated Lattices}\label{s:integral} 

The results of the preceding sections provide strong evidence of the importance of the notion of semilinearity in the study of Hamiltonian and nilpotent prelinear cancellative varieties. The present section is concerned with varieties of semilinear cancellative integral residuated lattices. It follows from standard facts in the theory of $\ell$-groups that a group embeds into a nilpotent \el group if and only if it admits a total order or, equivalently, if and only if it is nilpotent and torsion-free (cf.\ \cite[Theorem 2.2.4]{BMR1977}). In view of Theorem~\ref{t:five}, it is natural to ask whether every nilpotent cancellative monoid with unique roots admits a residuated total order. We provide a partial answer to this question, and show that any finitely generated free object in the quasivariety of nilpotent cancellative monoids with unique roots admits an integral total order (Theorem~\ref{t:orderonfree}). 

We call a total order $\le$ on a monoid (not necessarily residuated) {\em integral} if the monoid identity is the greatest element with respect to $\le$. We say that a poset $\m{P} = \langle P, \le \rangle$ satisfies the {\em ascending chain condition} (ACC) if $\m{P}$ does not contain any infinite ascending chains. Note that by (\ref{eq:residual2}), any total order on a monoid $\m{M}$ that satisfies the ACC is a residuated. 

\begin{lemma}\label{l:higman}
Every integral total order on a finitely generated monoid is residuated.
\end{lemma}

\begin{proof}
Let $\m{M}$ be a monoid generated by $n$ elements, and set $\le$ to be an integral total order on $\m{M}$. Then, there exists a surjective monoid homomorphism $\varphi$ from the free monoid $\m{M}(x_1, \ldots, x_n) = \m{M}(n)$ to $\m{M}$. We show that $\langle \m{M}, \le \rangle$ satisfies the ACC. Suppose that
\[
m_0 < m_1 < m_2 < \ldots < m_i < \ldots,
\]
is an infinite ascending chain in $\langle \m{M}, \le \rangle$. As $\varphi$ is onto, $\varphi^{-1}[\{m_i\}] \neq \emptyset$, for all $i \in \Z^+$. Consider
\[
\{t_i = f(\varphi^{-1}[\{m_i\}]) \mid i \in \Z^+\},
\]
where $f \colon \Z^+ \to \bigcup_{i \in \Z^+}\varphi^{-1}[\{m_i\}]$ is a choice function. Then, $\{t_i\}$ is an infinite sequence of words over the finite alphabet $\{x_1, \ldots, x_n\}$. By Higman's Lemma~\cite{MR0049867}, there must exist indices $i < j$ such that $t_i$ can be obtained from $t_j$ by deleting some symbols: e.g.,
\[
t_i = x_1\cdots \,x_k \quad\text{ and }\quad t_j=s_{j_0}x_{1}s_{j_1}\cdots \,x_{k}s_{j_k},
\]
where $s_{j_0}, s_{j_1}, \dots, s_{j_k}$ are arbitrary words in $\m{M}(n)$. Then, $\varphi(t_i) = m_i < m_j =\varphi(t_j)$, which entails
\[
\varphi(x_1)\cdots\varphi(x_k) < \varphi(s_{j_0})\varphi(x_{1})\varphi(s_{j_1})\cdots \,\varphi(x_{k})\varphi(s_{j_k}).
\]
This is a contradiction, since for all $a, b \in M$, $ab \le a, b$ due to the integrality of the order $\le$. Therefore, $\langle \m{M}, \le \rangle$ satisfies the ACC, and hence it is residuated.
\end{proof}

Let $\Icansemrl$ denote the variety of semilinear cancellative integral residuated lattices, and $\vty{V}$ be any monoid-subvariety of $\Icansemrl$. Observe that, as the variety $\vty{V}$ is defined relative to $\Icansemrl$ by a set $\Sigma$ of monoid equations, every member of the quasivariety $\vty{M}(\vty{V})$ satisfies $\Sigma$.

\begin{lemma}\label{l:orderonquasi}
For any monoid-subvariety $\vty{V}$ of $\Icansemrl$, every finitely generated monoid in the quasivariety $\vty{M}(\vty{V})$ is the monoid reduct of a totally ordered member of $\vty{V}$. 
\end{lemma}

\begin{proof}
Let $\m{M}$ be a finitely generated member of $\vty{M}(\vty{V})$ and a submonoid of a member $\m{L}$ of $\vty{V}$. 
Since $\vty{V}$ is semilinear, $\m{L}$ is the subdirect product of cancellative integral residuated chains $\m{T}_i$, $i \in I$. Let $\le$ be a well-order on $I$, and for $a=(a_i)_{i\in I}$, $b=(b_i)_{i \in I} \in L$, set 
\[
a \trianglelefteq b\quad\text{ iff }\quad a = b\,\,\text{ or }\,\,(a_j < b_j,\text{ where }j = \min\{i \in I \mid a_i \ne b_i\}).
\] 
We claim that $\trianglelefteq$ is an integral total order on $\m{L}$ extending its lattice order. Indeed, let $a, b, c$ be elements of $L$ such that $a \triangleleft b$. Then $a_j < b_j$ for $j = \min\{i \in I \mid a_i \ne b_i\}$. By cancellativity, $a_jc_j < b_jc_j$ (resp.\ $c_ja_j < c_jb_j$), and hence, $ac \triangleleft bc$ (resp.\ $ca \triangleleft cb$). The restriction of the total order $\trianglelefteq$ to the finitely generated monoid $\m{M}$ is residuated by Lemma~\ref{l:higman}. Moreover, as $\vty{V}$ is a monoid-subvariety of $\Icansemrl$ and $\m{M}$ is a submonoid of $\m{L}$, $\langle \m{M}, \le \rangle$ is a member of $\vty{V}$.
\end{proof}

Every variety of semilinear residuated lattices is generated by its finitely generated totally ordered members. In the case of monoid-subvarieties of $\Icansemrl$, we have the following stronger result.

\begin{theorem}\label{t:fgmongen}
Every monoid-subvariety of $\Icansemrl$ is generated by the class of residuated chains whose monoid reducts are finitely generated monoids. 
\end{theorem}

\begin{proof}
For any monoid-subvariety $\vty{V}$ of $\Icansemrl$, we show that an equation $t_1 \eq t_2$ that fails in $\vty{V}$ necessarily fails in a $\vty{V}$-chain whose monoid reduct is a finitely generated monoid. Let $t_1$ and $t_2$ be two residuated lattice terms such that 
\[
\nu(t_1)  \neq  \nu(t_2),
\] 
under the evaluation $\nu$ from the term algebra into the finitely generated residuated chain $\m{T}$. Let ${\sf s}(t_1)$ and ${\sf s}(t_2)$ denote, respectively, the set of all subterms of $t_1$ and the set of all subterms of $t_2$. Let $\m{M}$ be the submonoid of $\m{T}$ generated by the finite set:
\[
\{\nu(u) \mid u \in {\sf s}(t_1) \cup {\sf s}(t_2)\},
\]
and consider the restriction $\le\upharpoonright_{M}$ of the order $\le$ from $\m{T}$ to $\m{M}$. By Lemma~\ref{l:higman}, $\langle \m{M}, \le\upharpoonright_{M} \rangle$ is an integral residuated lattice, which is a submonoid and a sublattice of $\m{T}$---although it need not be a substructure, since residuals might not be preserved. Consider the evaluation $\mu$ from the term algebra into $\m{M}$, defined by $\mu(x) = \nu(x)$ for each variable $x$. We show that $\mu(u) = \nu(u)$, for every $u \in {\sf s}(t_1) \cup {\sf s}(t_2)$, by induction on the structure of $u$. The base case is trivial, since it follows by the definition of $\mu$. The cases involving monoid operation ($u = u_1 \cdot u_2$), and the lattice operations  ($u = u_1 \mt u_2$ or  $u = u_1 \jn u_2$) follows from the fact that $\m{M}$ is a submonoid and a sublattice of $\m{T}$. Suppose $u= u_1 \ld u_2$. It suffices to show 
\[
\mu(u_1) \ld_{\m{M}}\, \mu(u_2) = \mu(u_1) \ld_{\m{T}}\, \mu(u_2).
\] 
By induction hypothesis, $\mu(u_1) = \nu(u_1)$ and $\mu(u_2) = \nu(u_2)$. Therefore, 
\[
\mu(u_1) \ld_{\m{T}}\, \mu(u_2) = \nu(u_1) \ld_{\m{T}}\, \nu(u_2) = \nu(u_1 \ld u_2) \in {M}.
\]
Hence, we can conclude that
\[
\mu(u)  = \mu(u_1) \ld_{\m{M}}\, \mu(u_2)  = \mu(u_1) \ld_{\m{T}}\, \mu(u_2) = \nu(u_1) \ld_{\m{T}}\, \nu(u_2) = \nu(u),
\]
as was to be shown. Therefore, $\mu(t_1) \neq \mu(t_2)$ in $\m{M}$, and $t_1 \eq t_2$ fails in $\langle \m{M}, \le\upharpoonright_{M} \rangle$.
\end{proof}

The class of $c$-nilpotent semilinear cancellative integral residuated lattices is a monoid\nbd{-}subvariety of $\Icansemrl$, which we denote by $\Ni$. Thus, Theorem~\ref{t:fgmongen} can be applied to $\Ni$ as a particular case.

We mention a result that relates free objects in $\vty{M}(\Nrl)$ to free objects in $\vty{M}(\Ni)$.

\begin{theorem}\label{t:orderonfree}
For any $c \in \N$, every finitely generated free object in the quasivariety $\vty{M}(\Nrl)$ admits an integral residuated total order. 
\end{theorem}

\begin{proof}
Let $\m{F}_c(n)$ be the free $c$-nilpotent group generated by the finite set $X=\{x_1, \ldots, x_n\}$. We consider a total order $\le$ on $\m{F}_c(n)$. This is possible since $\m{F}_c(n)$ is torsion-free. Let ${\m S}_c(X^\delta)$ be the submonoid of $\m{F}_c(n)$ generated by $X^\delta=\{x_1^{\delta_1}, \ldots, x_n^{\delta_n}\}$, with $\delta_i \in \{-1,1\}$ and $x_i^{\delta_i} < \ut$ for each $i \in \{1, \ldots, n\}$. The restriction of the total order $\le$ to ${\m S}_c(X^\delta)$ induces an integral residuated total order on ${\m S}_c(X^\delta)$, by Lemma~\ref{l:higman}. Now, we conclude by observing that ${\m S}_c(X^\delta)$ is isomorphic to ${\m S}_c(X)$, and hence to $\m{M}_c(n)$ by Proposition~\ref{prop:fisleftquot}. For this, it suffices to consider the unique group homomorphism 
\[
\alpha \colon \m{F}_c(n) \to \m{F}_c(n),
\]
extending the map $x_i \mapsto x_i^{\delta_i}$. This is an automorphism of $\m{F}_c(n)$, whose restriction to ${\m S}_c(X)$ is a monoid isomorphism onto ${\m S}_c(X^\delta)$.
\end{proof}

\noindent By Theorem~\ref{t:orderonfree}, the free object over a finite set $X$ in $\vty{M}(\Ni)$ coincides with the free object $\m{M}_c(X)$ in the quasivariety $\vty{M}(\Nrl)$.

\begin{example}
For any $2 \le c \in \N$, we can find a negative cone of a $c$-nilpotent $\ell$-group ${\m G}$ that separates the variety of $c$-nilpotent cancellative integral residuated lattices from the variety of $(c-1)$-nilpotent cancellative integral residuated lattices. It suffices to take a $c$\nbd{-}nilpotent $\ell$-group ${\m G}$ that separates the variety of $c$-nilpotent $\ell$-groups from the variety of $(c-1)$-nilpotent $\ell$-groups. 
Given that any $\mathsf{L}_c$ is a monoid equation, the negative cone of ${\m G}$ is a $c$-nilpotent residuated lattice. Moreover, given that ${\m G}$ is not $(c-1)$\nbd{-}nilpotent, its negative cone does not satisfy $\mathsf{L}_{(c-1)}$ either. This follows from Proposition~\ref{prop:nilp}, Theorem~\ref{t:catequiv}, and from the fact that an $\ell$-group is the group of quotients of its negative cone (see, e.g., \cite[Proposition 4.1]{MR1304052}). For a related result, see~\cite[Corollary 6.3]{bahls2003cancellative}, where an isomorphism between varieties of $\ell$-groups and varieties of negative cones is exhibited.
\end{example}

We conclude by mentioning an interesting question that is left open in this paper:
Is it possible to obtain, analogously to Theorem~\ref{t:five}, a characterization of the submonoids of the negative cones of $c$-nilpotent $\ell$-groups? Note that a monoid $\m{M}$ can be embedded into the negative cone of an Abelian $\ell$\nbd{-}group if and only if $\m{M}$ is commutative, cancellative, has unique roots, and (*) does not contain any (non\nbd{-}trivial) invertible element. It is clear that (*) is necessary. To see that it suffices, let $\m{M}$ be a submonoid of a torsion-free Abelian group $\m{G}$ such that $M \cap \iv{M} = \{\ut\}$, where $\iv{M} = \{\iv{a} \mid a \in M\}$. Then, $\m{M}$ is the negative cone of a partial order on $\m{G}$ (see, e.g.,~\cite[Ch.\ II, Theorem 2]{MR0171864}). Since every partial order on a torsion-free Abelian group $\m{G}$ extends to a total order (see, e.g.,~\cite[Ch.\ III, Corollary 13]{MR0171864}), $\m{M}$ can be extended to the negative cone of a total order on $\m{G}$, and hence can be embedded into the negative cone of an Abelian totally ordered group. For non-commutative \el groups the condition (*) does not suffice. For instance for $c \ge 2$, a submonoid of a torsion-free $c$-nilpotent group $\m{G}$ satisfying (*) is, in general, the negative cone of a partial order on $G$ compatible only with right multiplication (for all $a, b, c \in G$, whenever $a \le b$, also $ac \le bc$), and it is not true that any such partial `right-invariant' order can be extended to a total order on $\m{G}$~\cite{MR1023968}. Therefore, it remains open to characterize those submonoids that can be embedded into the negative cone of a $c$-nilpotent $\ell$-group---equivalently, the submonoids of $c$-nilpotent cancellative integral residuated lattices (cf.\ Theorem~\ref{t:five}).

\vskip0.5cm
\section*{Acknowledgements}

We are grateful to Adam P\v{r}enosil for providing the proof of Lemma~\ref{l:higman}. 

\vskip0.75cm


\end{document}